\documentclass[12pt]{amsart}       
\usepackage{txfonts}
\usepackage{amssymb}
\usepackage{eucal}
\usepackage{graphicx}
\usepackage{amsmath}
\usepackage{amscd}
\usepackage[all]{xy}           
\usepackage{tikz}
\usepackage{amsfonts,latexsym}
\usepackage{xspace}
\usepackage{epsfig}
\usepackage{float}
\usepackage{axodraw}
\usepackage{color}
\usepackage{fancybox}
\usepackage{colordvi}
\usepackage{multicol}
\usepackage{eucal}

\usepackage[all]{xy}           

\usepackage{xspace}

\usepackage[colorlinks,final,backref=page,hyperindex,hypertex]{hyperref}
\usepackage[active]{srcltx} 

\newtheorem{theorem}{Theorem}[section]
\newtheorem{prop}[theorem]{Proposition}
\theoremstyle{definition}
\newtheorem{defn}[theorem]{Definition}
\newtheorem{lemma}[theorem]{Lemma}

\newtheorem{prop-def}{Proposition-Definition}[section]
\newtheorem{coro-def}{Corollary-Definition}[section]

\newtheorem{exam}{Example}[section]

\newcommand{\nc}{\newcommand}
\nc{\tred}[1]{\textcolor{red}{#1}}
\nc{\tblue}[1]{\textcolor{blue}{#1}}
\nc{\tgreen}[1]{\textcolor{green}{#1}}
\nc{\tpurple}[1]{\textcolor{purple}{#1}}
\nc{\btred}[1]{\textcolor{red}{\bf #1}}
\nc{\btblue}[1]{\textcolor{blue}{\bf #1}}
\nc{\btgreen}[1]{\textcolor{green}{\bf #1}}
\nc{\btpurple}[1]{\textcolor{purple}{\bf #1}}
\nc{\NN}{{\mathbb N}}
\nc{\ncsha}{{\mbox{\cyr X}^{\mathrm NC}}} \nc{\ncshao}{{\mbox{\cyr
X}^{\mathrm NC}_0}}

\renewcommand{\frak}{\mathfrak}

\newcommand{\efootnote}[1]{}

\renewcommand{\textbf}[1]{}

\newcommand{\delete}[1]{}

\nc{\mlabel}[1]{\label{#1}}  
\nc{\mcite}[1]{\cite{#1}}  
\nc{\mref}[1]{\ref{#1}}  
\nc{\mbibitem}[1]{\bibitem{#1}} 

\delete{
\nc{\mlabel}[1]{\label{#1}  
{\hfill \hspace{1cm}{\small\tt{{\ }\hfill(#1)}}}}
\nc{\mcite}[1]{\cite{#1}{\small{\tt{{\ }(#1)}}}}  
\nc{\mref}[1]{\ref{#1}{{\tt{{\ }(#1)}}}}  
\nc{\mbibitem}[1]{\bibitem[\bf #1]{#1}} 
}

\nc{\vecs}{\vec{s}}
\nc{\gen}[1]{{\langle {#1}\rangle}}
\nc{\rsubtree}{real subtree\xspace}
\nc{\lsubtree}{letter subtree\xspace}
\nc{\rsubtrees}{real subtrees\xspace}
\nc{\lsubtrees}{letter subtrees\xspace}
\nc{\closure}{closure\xspace}
\nc{\closures}{closures\xspace}
\nc{\cl}{\mathrm{cl}}

\nc{\vsubtree}{virtual subtree\xspace}
\nc{\vsubtrees}{virtual subtrees\xspace}
\nc{\vsubforest}{virtual subforest\xspace}
\nc{\vsubforests}{virtual subforests\xspace}

\nc{\mnoindent}{\smallskip \noindent}

\nc{\mo}{m}
\nc{\barot}{\overline{\otimes}}

\nc{\opa}{\ast} \nc{\opb}{\odot} \nc{\op}{\bullet} \nc{\pa}{\frakL}
\nc{\arr}{\rightarrow} \nc{\lu}[1]{(#1)} \nc{\mult}{\mrm{mult}}
\nc{\diff}{\mathfrak{Diff}}
\nc{\opc}{\sharp}\nc{\opd}{\natural}
\nc{\ope}{\circ}
\nc{\dpt}{\mathrm{d}}
\nc{\diam}{alternating\xspace}
\nc{\Diam}{Alternating\xspace}
\nc{\cdiam}{canonical alternating\xspace}
\nc{\Cdiam}{Canonical alternating\xspace}
\nc{\AW}{\mathcal{A}}
\nc{\mrbo}{modified RBO\xspace }
\nc{\ari}{\mathrm{ar}}

\nc{\lef}{\mathrm{lef}}

\nc{\Sh}{\mathrm{ST}}

\nc{\Cr}{\mathrm{Cr}}

\nc{\st}{{Schr\"oder tree}\xspace}
\nc{\sts}{{Schr\"oder trees}\xspace}

\nc{\vertset}{\Omega} 

\nc{\pb}{{\mathrm{pb}}}
\nc{\Lf}{{\mathrm{Lf}}}

\nc{\lft}{{left tree}\xspace}
\nc{\lfts}{{left trees}\xspace}

\nc{\fat}{{fundamental averaging tree}\xspace}

\nc{\fats}{{fundamental averaging trees}\xspace}
\nc{\avt}{\mathrm{Avt}}

\nc{\rass}{{\mathit{RAss}}}

\nc{\aass}{{\mathit{AAss}}}

\nc{\vin}{{\mathrm Vin}}    
\nc{\lin}{{\mathrm Lin}}    
\nc{\inv}{\mathrm{I}n}
\nc{\gensp}{V} 
\nc{\genbas}{\mathcal{V}} 
\nc{\bvp}{V_P}     
\nc{\gop}{{\,\omega\,}}     

\nc{\bin}[2]{ (_{\stackrel{\scs{#1}}{\scs{#2}}})}  
\nc{\binc}[2]{ \left (\!\! \begin{array}{c} \scs{#1}\\
    \scs{#2} \end{array}\!\! \right )}  
\nc{\bincc}[2]{  \left ( {\scs{#1} \atop
    \vspace{-1cm}\scs{#2}} \right )}  
\nc{\bs}{\bar{S}} \nc{\cosum}{\sqsubset} \nc{\la}{\longrightarrow}
\nc{\rar}{\rightarrow} \nc{\dar}{\downarrow} \nc{\dprod}{**}
\nc{\dap}[1]{\downarrow \rlap{$\scriptstyle{#1}$}}
\nc{\md}{\mathrm{dth}} \nc{\uap}[1]{\uparrow
\rlap{$\scriptstyle{#1}$}} \nc{\defeq}{\stackrel{\rm def}{=}}
\nc{\disp}[1]{\displaystyle{#1}} \nc{\dotcup}{\
\displaystyle{\bigcup^\bullet}\ } \nc{\gzeta}{\bar{\zeta}}
\nc{\hcm}{\ \hat{,}\ } \nc{\hts}{\hat{\otimes}}
\nc{\free}[1]{\bar{#1}}
\nc{\uni}[1]{\tilde{#1}} \nc{\hcirc}{\hat{\circ}} \nc{\lleft}{[}
\nc{\lright}{]} \nc{\lc}{\lfloor} \nc{\rc}{\rfloor}
\nc{\curlyl}{\left \{ \begin{array}{c} {} \\ {} \end{array}
    \right .  \!\!\!\!\!\!\!}
\nc{\curlyr}{ \!\!\!\!\!\!\!
    \left . \begin{array}{c} {} \\ {} \end{array}
    \right \} }
\nc{\longmid}{\left | \begin{array}{c} {} \\ {} \end{array}
    \right . \!\!\!\!\!\!\!}
\nc{\onetree}{\bullet} \nc{\ora}[1]{\stackrel{#1}{\rar}}
\nc{\ola}[1]{\stackrel{#1}{\la}}
\nc{\ot}{\otimes} \nc{\mot}{{{\boxtimes\,}}}
\nc{\otm}{\overline{\boxtimes}} \nc{\sprod}{\bullet}
\nc{\scs}[1]{\scriptstyle{#1}} \nc{\mrm}[1]{{\rm #1}}
\nc{\margin}[1]{\marginpar{\rm #1}}   
\nc{\dirlim}{\displaystyle{\lim_{\longrightarrow}}\,}
\nc{\invlim}{\displaystyle{\lim_{\longleftarrow}}\,}
\nc{\mvp}{\vspace{0.3cm}} \nc{\tk}{^{(k)}} \nc{\tp}{^\prime}
\nc{\ttp}{^{\prime\prime}} \nc{\svp}{\vspace{2cm}}
\nc{\vp}{\vspace{8cm}} \nc{\proofbegin}{\noindent{\bf Proof: }}
\nc{\proofend}{$\blacksquare$ \vspace{0.3cm}}
\nc{\modg}[1]{\!<\!\!{#1}\!\!>}
\nc{\intg}[1]{F_C(#1)} \nc{\lmodg}{\!
<\!\!} \nc{\rmodg}{\!\!>\!}
\nc{\cpi}{\widehat{\Pi}}
\nc{\sha}{{\mbox{\cyr X}}}  
\nc{\ssha}{{\mbox{\cyrs X}}} 
\nc{\shpr}{\diamond}    
\nc{\shp}{\ast} \nc{\shplus}{\shpr^+}
\nc{\shprc}{\shpr_c}    
\nc{\msh}{\ast} \nc{\zprod}{m_0} \nc{\oprod}{m_1}
\nc{\vep}{\varepsilon} \nc{\labs}{\mid\!} \nc{\rabs}{\!\mid}
\nc{\sqmon}[1]{\langle #1\rangle}

\nc{\mmbox}[1]{\mbox{\ #1\ }} \nc{\dep}{\mrm{dep}}
\nc{\bre}{\mrm{bre}}
\nc{\fp}{\mrm{FP}}
\nc{\rchar}{\mrm{char}} \nc{\End}{\mrm{End}} \nc{\Fil}{\mrm{Fil}}
\nc{\Mor}{Mor\xspace} \nc{\gmzvs}{gMZV\xspace}
\nc{\gmzv}{gMZV\xspace} \nc{\mzv}{MZV\xspace}
\nc{\mzvs}{MZVs\xspace} \nc{\Hom}{\mrm{Hom}} \nc{\id}{\mrm{id}}
\nc{\im}{\mrm{im}} \nc{\incl}{\mrm{incl}} \nc{\map}{\mrm{Map}}
\nc{\mchar}{\rm char} \nc{\nz}{\rm NZ} \nc{\supp}{\mathrm Supp}

\nc{\Alg}{\mathbf{Alg}} \nc{\Bax}{\mathbf{Bax}} \nc{\bff}{\mathbf f}
\nc{\bfk}{{\bf k}} \nc{\bfone}{{\bf 1}} \nc{\bfx}{\mathbf x}
\nc{\bfy}{\mathbf y}
\nc{\base}[1]{\bfone^{\otimes ({#1}+1)}} 
\nc{\Cat}{\mathbf{Cat}}

\nc{\detail}{\marginpar{\bf More detail}
    \noindent{\bf Need more detail!}
    \svp}
\nc{\Int}{\mathbf{Int}} \nc{\Mon}{\mathbf{Mon}}
\nc{\rbtm}{{shuffle }} \nc{\rbto}{{Rota-Baxter }}
\nc{\remarks}{\noindent{\bf Remarks: }} \nc{\Rings}{\mathbf{Rings}}
\nc{\Sets}{\mathbf{Sets}} \nc{\wtot}{\widetilde{\odot}}
\nc{\wast}{\widetilde{\ast}} \nc{\bodot}{\bar{\odot}}
\nc{\bast}{\bar{\ast}} \nc{\hodot}[1]{\odot^{#1}}
\nc{\hast}[1]{\ast^{#1}} \nc{\mal}{\mathcal{O}}
\nc{\tet}{\tilde{\ast}} \nc{\teot}{\tilde{\odot}}
\nc{\oex}{\overline{x}} \nc{\oey}{\overline{y}}
\nc{\oez}{\overline{z}} \nc{\oef}{\overline{f}}
\nc{\oea}{\overline{a}} \nc{\oeb}{\overline{b}}
\nc{\weast}[1]{\widetilde{\ast}^{#1}}
\nc{\weodot}[1]{\widetilde{\odot}^{#1}} \nc{\hstar}[1]{\star^{#1}}
\nc{\lae}{\langle} \nc{\rae}{\rangle}
\nc{\lf}{\lfloor}
\nc{\rf}{\rfloor}

\nc{\QQ}{{\mathbb Q}}
\nc{\RR}{{\mathbb R}} \nc{\ZZ}{{\mathbb Z}}

\nc{\cala}{{\mathcal A}} \nc{\calb}{{\mathcal B}}
\nc{\calc}{{\mathcal C}}
\nc{\cald}{{\mathcal D}} \nc{\cale}{{\mathcal E}}
\nc{\calf}{{\mathcal F}} \nc{\calg}{{\mathcal G}}
\nc{\calh}{{\mathcal H}} \nc{\cali}{{\mathcal I}}
\nc{\call}{{\mathcal L}} \nc{\calm}{{\mathcal M}}
\nc{\caln}{{\mathcal N}} \nc{\calo}{{\mathcal O}}
\nc{\calp}{{\mathcal P}} \nc{\calr}{{\mathcal R}}
\nc{\cals}{{\mathcal S}} \nc{\calt}{{\mathcal T}}
\nc{\calu}{{\mathcal U}} \nc{\calw}{{\mathcal W}} \nc{\calk}{{\mathcal K}}
\nc{\calx}{{\mathcal X}} \nc{\CA}{\mathcal{A}}

\nc{\fraka}{{\mathfrak a}} \nc{\frakA}{{\mathfrak A}}
\nc{\frakb}{{\mathfrak b}} \nc{\frakB}{{\mathfrak B}}
\nc{\frakD}{{\mathfrak D}} \nc{\frakF}{\mathfrak{F}}
\nc{\frakf}{{\mathfrak f}} \nc{\frakg}{{\mathfrak g}}
\nc{\frakH}{{\mathfrak H}} \nc{\frakL}{{\mathfrak L}}
\nc{\frakM}{{\mathfrak M}} \nc{\bfrakM}{\overline{\frakM}}
\nc{\frakm}{{\mathfrak m}} \nc{\frakP}{{\mathfrak P}}
\nc{\frakN}{{\mathfrak N}} \nc{\frakp}{{\mathfrak p}}
\nc{\frakS}{{\mathfrak S}} \nc{\frakT}{\mathfrak{T}}
\nc{\frakx}{\mathfrak{x}}  \nc{\frakX}{{\mathfrak X}}

\nc{\BS}{\mathbb{S
}}

\font\cyr=wncyr10 \font\cyrs=wncyr7
\nc{\li}[1]{\textcolor{red}{#1}}
\nc{\lir}[1]{\textcolor{red}{Li:#1}}

\nc{\xigou}[1]{\textcolor{blue}{Xigou: #1}}

\nc{\UN}{U_{N}}
\nc{\FN}{F_{\mathrm M}}
\nc{\altx}{\Lambda}
\nc{\spr}{\cdot}
\nc{\rts}{\stackrel{\rightarrow}{\shpr}}
\nc{\ox}{\overline{\frak x}}
\nc{\oX}{\overline{X}}
\nc{\dia}{\diamond_a}
\nc{\dirr}{\diamond_R}
\nc{\lam}{\lambda}
\nc{\ovf}{\overline}
\nc{\sj}{\Delta_a}
\nc{\eb}{\epsilon_a}
\nc{\iu}{\iota}
\nc{\pq}{\preceq}
\nc{\di}{\bullet}
\nc{\wi}{\widetilde{\Delta_a}}

\def\ta1{{\scalebox{0.25}{ 
\begin{picture}(12,12)(38,-38)
\SetWidth{0.5}  \Vertex(45,-33){5.66}
\end{picture}}}}

\def\tb2{{\scalebox{0.25}{ 
\begin{picture}(12,42)(38,-38)
\SetWidth{0.5}  \Vertex(45,-3){5.66}
\SetWidth{1.0} \Line(45,-3)(45,-33) \SetWidth{0.5}
\Vertex(45,-33){5.66}
\end{picture}}}\,}

\def\tc3{{\scalebox{0.25}{ 
\begin{picture}(12,72)(38,-38)
\SetWidth{0.5}  \Vertex(45,27){5.66}
\SetWidth{1.0} \Line(45,27)(45,-3) \SetWidth{0.5}
\Vertex(45,-33){5.66} \SetWidth{1.0} \Line(45,-3)(45,-33)
\SetWidth{0.5} \Vertex(45,-3){5.66}
\end{picture}}}}

\def\td31{{\scalebox{0.25}{ 
\begin{picture}(42,42)(23,-38)
\SetWidth{0.5}  \Vertex(45,-3){5.66}
\Vertex(30,-33){5.66} \Vertex(60,-33){5.66} \SetWidth{1.0}
\Line(45,-3)(30,-33) \Line(60,-33)(45,-3)
\end{picture}}}}

\def\xtd31{{\scalebox{0.35}{ 
\begin{picture}(70,42)(13,-35)
\SetWidth{0.5}  \Vertex(45,-3){5.66}
\Vertex(30,-33){5.66} \Vertex(60,-33){5.66} \SetWidth{1.0}
\Line(45,-3)(30,-33) \Line(60,-33)(45,-3)
\put(38,-38){\em \huge x}
\end{picture}}}}

\def\x2td31{{\scalebox{0.35}{ 
\begin{picture}(70,42)(13,-35)
\SetWidth{0.5}  \Vertex(45,-3){5.66}
\Vertex(28,-33){5.66} \Vertex(62,-33){5.66} \SetWidth{1.0}
\Line(45,-3)(28,-33) \Line(62,-33)(45,-3)
\put(35,-38){\em \huge $x_2$}
\end{picture}}}}

\def\ytd31{{\scalebox{0.35}{ 
\begin{picture}(70,42)(13,-35)
\SetWidth{0.5}  \Vertex(45,-3){5.66}
\Vertex(30,-33){5.66} \Vertex(60,-33){5.66} \SetWidth{1.0}
\Line(45,-3)(30,-33) \Line(60,-33)(45,-3)
\put(38,-38){\em \huge y}
\end{picture}}}}

\def\xldec41r{{\scalebox{0.35}{ 
\begin{picture}(70,42)(13,-45)

\SetWidth{0.5} \Vertex(45,-3){5.66}
\Vertex(30,-33){5.66} \Vertex(60,-33){5.66}
\Vertex(60,-63){5.66}
\SetWidth{1.0}
\Line(45,-3)(30,-33) \Line(60,-33)(45,-3)
\Line(60,-33)(60,-63)
\put(38,-38){\em \huge x}

\end{picture}}}}

\def\xyldec43{{\scalebox{0.35}{ 
\begin{picture}(70,62)(13,-25)

\SetWidth{0.5} \Vertex(45,-3){5.66}
\Vertex(15,-33){5.66} \Vertex(45,-38){5.66}
\Vertex(75,-33){5.66}
\SetWidth{1.0}
\Line(45,-3)(15,-33) \Line(45,-3)(45,-38)
\Line(45,-3)(74,-33)
\put(25,-33){\em\huge x}
\put(50,-33){\em\huge y}
\end{picture}}}}

\def\te4{{\scalebox{0.25}{ 
\begin{picture}(12,102)(38,-8)
\SetWidth{0.5}  \Vertex(45,57){5.66}
\Vertex(45,-3){5.66} \Vertex(45,27){5.66} \Vertex(45,87){5.66}
\SetWidth{1.0} \Line(45,57)(45,27) \Line(45,-3)(45,27)
\Line(45,57)(45,87)
\end{picture}}}}

\def\tf41{{\scalebox{0.25}{ 
\begin{picture}(42,72)(38,-8)
\SetWidth{0.5}  \Vertex(45,27){5.66}
\Vertex(45,-3){5.66} \SetWidth{1.0} \Line(45,27)(45,-3)
\SetWidth{0.5} \Vertex(60,57){5.66} \SetWidth{1.0}
\Line(45,27)(60,57) \SetWidth{0.5} \Vertex(75,27){5.66}
\SetWidth{1.0} \Line(75,27)(60,57)
\end{picture}}}}

\def\tg42{{\scalebox{0.25}{ 
\begin{picture}(42,72)(8,-8)
\SetWidth{0.5}  \Vertex(45,27){5.66}
\Vertex(45,-3){5.66} \SetWidth{1.0} \Line(45,27)(45,-3)
\SetWidth{0.5} \Vertex(15,27){5.66} \Vertex(30,57){5.66}
\SetWidth{1.0} \Line(15,27)(30,57) \Line(45,27)(30,57)
\end{picture}}}}

\def\th43{{\scalebox{0.25}{ 
\begin{picture}(42,42)(8,-8)
\SetWidth{0.5}  \Vertex(45,-3){5.66}
\Vertex(15,-3){5.66} \Vertex(30,27){5.66} \SetWidth{1.0}
\Line(15,-3)(30,27) \Line(45,-3)(30,27) \Line(30,27)(30,-3)
\SetWidth{0.5} \Vertex(30,-3){5.66}
\end{picture}}}}

\def\thII43{{\scalebox{0.25}{ 
\begin{picture}(72,57) (68,-128)
    \SetWidth{0.5}

    \Vertex(105,-78){5.66}
    \SetWidth{1.5}
    \Line(105,-78)(75,-123)
    \Line(105,-78)(105,-123)
    \Line(105,-78)(135,-123)
    \SetWidth{0.5}
    \Vertex(75,-123){5.66}
    \Vertex(105,-123){5.66}
    \Vertex(135,-123){5.66}
  \end{picture}
  }}}

\def\thj44{{\scalebox{0.25}{ 
\begin{picture}(42,72)(8,-8)
\SetWidth{0.5}  \Vertex(30,57){5.66}
\SetWidth{1.0} \Line(30,57)(30,27) \SetWidth{0.5}
\Vertex(30,27){5.66} \SetWidth{1.0} \Line(45,-3)(30,27)
\SetWidth{0.5} \Vertex(45,-3){5.66} \Vertex(15,-3){5.66}
\SetWidth{1.0} \Line(15,-3)(30,27)
\end{picture}}}}

\def\xthj44{{\scalebox{0.35}{ 
\begin{picture}(42,72)(8,-8)
\SetWidth{0.5}  \Vertex(30,57){5.66}
\SetWidth{1.0} \Line(30,57)(30,27) \SetWidth{0.5}
\Vertex(30,27){5.66} \SetWidth{1.0} \Line(45,-3)(30,27)
\SetWidth{0.5} \Vertex(45,-3){5.66} \Vertex(15,-3){5.66}
\SetWidth{1.0} \Line(15,-3)(30,27)
\put(25,-3){\em\huge x}
\end{picture}}}}

\def\ti5{{\scalebox{0.25}{ 
\begin{picture}(12,132)(23,-8)
\SetWidth{0.5}  \Vertex(30,117){5.66}
\SetWidth{1.0} \Line(30,117)(30,87) \SetWidth{0.5}
\Vertex(30,87){5.66} \Vertex(30,57){5.66} \Vertex(30,27){5.66}
\Vertex(30,-3){5.66} \SetWidth{1.0} \Line(30,-3)(30,27)
\Line(30,27)(30,57) \Line(30,87)(30,57)
\end{picture}}}}

\def\tj51{{\scalebox{0.25}{ 
\begin{picture}(42,102)(53,-38)
\SetWidth{0.5}  \Vertex(61,27){4.24}
\SetWidth{1.0} \Line(75,57)(90,27) \Line(60,27)(75,57)
\SetWidth{0.5} \Vertex(90,-3){5.66} \Vertex(60,27){5.66}
\Vertex(75,57){5.66} \Vertex(90,-33){5.66} \SetWidth{1.0}
\Line(90,-33)(90,-3) \Line(90,-3)(90,27) \SetWidth{0.5}
\Vertex(90,27){5.66}
\end{picture}}}}

\def\tk52{{\scalebox{0.25}{ 
\begin{picture}(42,102)(23,-8)
\SetWidth{0.5}  \Vertex(60,57){5.66}
\Vertex(45,87){5.66} \SetWidth{1.0} \Line(45,87)(60,57)
\SetWidth{0.5} \Vertex(30,57){5.66} \SetWidth{1.0}
\Line(30,57)(45,87) \SetWidth{0.5} \Vertex(30,-3){5.66}
\SetWidth{1.0} \Line(30,-3)(30,27) \SetWidth{0.5}
\Vertex(30,27){5.66} \SetWidth{1.0} \Line(30,57)(30,27)
\end{picture}}}}

\def\tl53{{\scalebox{0.25}{ 
\begin{picture}(42,102)(8,-8)
\SetWidth{0.5}  \Vertex(30,57){5.66}
\Vertex(30,27){5.66} \SetWidth{1.0} \Line(30,57)(30,27)
\SetWidth{0.5} \Vertex(30,87){5.66} \SetWidth{1.0}
\Line(30,27)(45,-3) \SetWidth{0.5} \Vertex(15,-3){5.66}
\SetWidth{1.0} \Line(15,-3)(30,27) \Line(30,57)(30,87)
\SetWidth{0.5} \Vertex(45,-3){5.66}
\end{picture}}}}

\def\tm54{{\scalebox{0.25}{ 
\begin{picture}(42,72)(8,-38)
\SetWidth{0.5}  \Vertex(30,-3){5.66}
\SetWidth{1.0} \Line(30,27)(30,-3) \Line(30,-3)(45,-33)
\SetWidth{0.5} \Vertex(15,-33){5.66} \SetWidth{1.0}
\Line(15,-33)(30,-3) \SetWidth{0.5} \Vertex(45,-33){5.66}
\SetWidth{1.0} \Line(30,-33)(30,-3) \SetWidth{0.5}
\Vertex(30,-33){5.66} \Vertex(30,27){5.66}
\end{picture}}}}

\def\tn55{{\scalebox{0.25}{ 
\begin{picture}(42,72)(8,-38)
\SetWidth{0.5}  \Vertex(15,-33){5.66}
\Vertex(45,-33){5.66} \Vertex(30,27){5.66} \SetWidth{1.0}
\Line(45,-33)(45,-3) \SetWidth{0.5} \Vertex(45,-3){5.66}
\Vertex(15,-3){5.66} \SetWidth{1.0} \Line(30,27)(45,-3)
\Line(15,-3)(30,27) \Line(15,-3)(15,-33)
\end{picture}}}}

\def\tp56{{\scalebox{0.25}{ 
\begin{picture}(66,111)(0,0)
\SetWidth{0.5}  \Vertex(30,66){5.66}
\Vertex(45,36){5.66} \SetWidth{1.0} \Line(30,66)(45,36)
\Line(15,36)(30,66) \SetWidth{0.5} \Vertex(30,6){5.66}
\Vertex(60,6){5.66} \SetWidth{1.0} \Line(60,6)(45,36)
\SetWidth{0.5}
\SetWidth{1.0} \Line(45,36)(30,6) \SetWidth{0.5}
\Vertex(15,36){5.66}
\end{picture}}}}

\def\tq57{{\scalebox{0.25}{ 
\begin{picture}(81,111)(0,0)
\SetWidth{0.5}  \Vertex(45,36){5.66}
\Vertex(30,6){5.66} \Vertex(60,6){5.66} \SetWidth{1.0}
\Line(60,6)(45,36) \SetWidth{0.5}
\SetWidth{1.0} \Line(45,36)(30,6) \SetWidth{0.5}
\Vertex(75,36){5.66} \SetWidth{1.0} \Line(45,36)(60,66)
\Line(60,66)(75,36) \SetWidth{0.5} \Vertex(60,66){5.66}
\end{picture}}}}

\def\tr58{{\scalebox{0.25}{ 
\begin{picture}(81,111)(0,0)
\SetWidth{0.5}  \Vertex(60,6){5.66}
\Vertex(75,36){5.66} \SetWidth{1.0} \Line(60,66)(75,36)
\SetWidth{0.5} \Vertex(60,66){5.66}
\SetWidth{1.0} \Line(60,36)(60,66) \Line(60,6)(60,36)
\SetWidth{0.5} \Vertex(60,36){5.66} \Vertex(45,36){5.66}
\SetWidth{1.0} \Line(60,66)(45,36)
\end{picture}}}}

\def\ap1{{\scalebox{0.35}{ 
\begin{picture}(70,42)(13,-35)
\SetWidth{0.5}  \Vertex(45,-3){5.66}
\Vertex(30,-33){5.66} \Vertex(60,-33){5.66} \SetWidth{1.0}
\Line(45,-3)(30,-33) \Line(60,-33)(45,-3)
\put(38,-38){\em \huge $\iu$}
\end{picture}}}}

\def\aq2{{\scalebox{0.35}{ 
\begin{picture}(70,42)(13,-35)
\SetWidth{0.5}  \Vertex(45,-3){5.66}
\Vertex(25,-33){5.66} \Vertex(65,-33){5.66} \SetWidth{1.0}
\Line(45,-3)(25,-33) \Line(65,-33)(45,-3)
\put(38,-38){\em \huge $x_2$}
\end{picture}}}}

\def\ao3{{\scalebox{0.35}{ 
\begin{picture}(42,72)(38,-8)
\SetWidth{0.5}  \Vertex(70,27){5.66}
\Vertex(45,-3){5.66} \SetWidth{1.0} \Line(45,-3)(70,27)
\SetWidth{0.5} \Vertex(70,58){5.66} \SetWidth{1.0}
\Line(70,27)(70,58) \SetWidth{0.5} \Vertex(95,-3){5.66}
\SetWidth{1.0} \Line(95,-3)(70,27)
\put(65,-5){\em \huge ${x_1}$}
\end{picture}}}}

\def\au4{{\scalebox{0.35}{ 
\begin{picture}(42,72)(38,-8)
\SetWidth{0.5}  \Vertex(40,27){5.66}
\Vertex(40,-3){5.66} \SetWidth{1.0} \Line(40,27)(40,-3)
\SetWidth{0.5} \Vertex(60,57){5.66} \SetWidth{1.0}
\Line(40,27)(60,57) \SetWidth{0.5} \Vertex(80,27){5.66}
\SetWidth{1.0} \Line(80,27)(60,57)
\put(50,20){\em \huge ${x_2}$}
\end{picture}}}}

\def\ay5{{\scalebox{0.35}{ 
\begin{picture}(70,72)(13,-58)

\SetWidth{0.5} \Vertex(50,-3){5.66}
\Vertex(30,-33){5.66} \Vertex(70,-33){5.66} \SetWidth{1.0}
\Line(50,-3)(30,-33) \Line(70,-33)(50,-3)
\put(40,-33){\em\huge ${x_2}$}
\SetWidth{0.5}
\Vertex(10,-63){5.66} \Vertex(50,-63){5.66} \SetWidth{1.0}
\Line(30,-33)(10,-63) \Line(30,-33)(50,-63)
\put(18,-63){\em\huge ${x_1}$}
\end{picture}}}}

\def\az6{{\scalebox{0.35}{ 
\begin{picture}(70,72)(13,-48)

\SetWidth{0.5} \Vertex(50,-3){5.66}
\Vertex(30,-33){5.66} \Vertex(70,-33){5.66} \SetWidth{1.0}
\Line(50,-3)(30,-33) \Line(70,-33)(50,-3)
\put(40,-33){\em\huge ${x_2}$}
\SetWidth{0.5}
\Vertex(10,-63){5.66} \Vertex(50,-63){5.66} \SetWidth{1.0}
\Line(30,-33)(10,-63) \Line(30,-33)(50,-63)
\put(18,-63){\em\huge $\iu$}
\end{picture}}}}

\def\as7{{\scalebox{0.35}{ 
\begin{picture}(70,72)(13,-48)

\SetWidth{0.5} \Vertex(50,-3){5.66}
\Vertex(30,-33){5.66} \Vertex(70,-33){5.66} \SetWidth{1.0}
\Line(50,-3)(30,-33) \Line(70,-33)(50,-3)
\put(40,-33){\em\huge $\iu$}
\SetWidth{0.5}
\Vertex(10,-63){5.66} \Vertex(50,-63){5.66} \SetWidth{1.0}
\Line(30,-33)(10,-63) \Line(30,-33)(50,-63)
\put(18,-63){\em\huge ${x_1}$}
\end{picture}}}}

\def\aw8{{\scalebox{0.35}{ 
\begin{picture}(70,72)(13,-48)

\SetWidth{0.5} \Vertex(50,-3){5.66}
\Vertex(30,-33){5.66} \Vertex(70,-33){5.66} \SetWidth{1.0}
\Line(50,-3)(30,-33) \Line(70,-33)(50,-3)
\put(40,-33){\em\huge $\iu$}
\SetWidth{0.5}
\Vertex(10,-63){5.66} \Vertex(50,-63){5.66} \SetWidth{1.0}
\Line(30,-33)(10,-63) \Line(30,-33)(50,-63)
\put(18,-63){\em\huge $\iu$}
\end{picture}}}}

\def\av9{{\scalebox{0.35}{ 
\begin{picture}(70,42)(13,-35)
\SetWidth{0.5}  \Vertex(45,-3){5.66}
\Vertex(20,-33){5.66} \Vertex(70,-33){5.66} \SetWidth{1.0}
\Line(45,-3)(20,-33) \Line(70,-33)(45,-3)
\put(38,-38){\em \huge ${x_2}$}
\end{picture}}}}

\def\ax10{{\scalebox{0.35}{ 
\begin{picture}(70,42)(13,-35)
\SetWidth{0.5}  \Vertex(45,-3){5.66}
\Vertex(20,-33){5.66} \Vertex(70,-33){5.66} \SetWidth{1.0}
\Line(45,-3)(20,-33) \Line(70,-33)(45,-3)
\put(38,-38){\em \huge ${x_1}$}
\end{picture}}}}
\begin{document}

\title[Hopf algebra, Rota-Baxter algebras and rooted trees]{Hopf algebra structure on free Rota-Baxter algebras by angularly decorated rooted trees}

\author{Xigou Zhang}
\address{Department of Mathematics, Jiangxi Normal University, Nanchang, Jiangxi 330022, China}
\email{xyzhang@jxnu.edu.cn}

\author{Anqi Xu}
\address{Department of Mathematics, Jiangxi Normal University, Nanchang, Jiangxi 330022, China}
\email{1147086708@qq.com}

\author{Li Guo}
\address{Department of Mathematics and Computer Science,
         Rutgers University,
         Newark, NJ 07102, USA}
\email{liguo@rutgers.edu}

\date{\today}
\begin{abstract}
By means of a new notion of subforests of an angularly decorated rooted forest, we give a combinatorial construction of a coproduct on the free Rota-Baxter algebra on angularly decorated rooted forests. We show that this coproduct equips the Rota-Baxter algebra with a bialgebra structure and further a Hopf algebra structure.
\end{abstract}

\subjclass[2010]
{16W99,16S10,17B38,16T05,05E16,16T30}

\keywords{Rota-Baxter algebra, angularly decoration, rooted forest, rooted tree, bialgebra, Hopf algebra, cocycle}

\maketitle

\tableofcontents

\setcounter{section}{0}

\allowdisplaybreaks

\section{Introduction}
\mlabel{sec:intr}

The study of rooted trees is important in combinatorics and has broad applications. Many algebraic structures have been equipped on rooted trees which give intuitive meaning to these abstract structures. Well-known examples of Hopf algebras on rooted trees include those of Connes-Kreimer, Loday-Ronco, Foissy-Holtkamp and Grassman-Larson~\mcite{CK,Fo1,Fo2,Hol,LR,GL}.

A major advantage of applying combinatorial objects and methods in algebra, especially in Hopf algebra, is that the algebraic operations can be described intuitively and explicitly. A prime example is the Connes-Kreimer Hopf algebra of rooted trees, as a baby model of the Hopf algebra of Feynman graphs arising from their study on renormalization of quantum field theory~\mcite{Kr,Kr2,DM}. Even though the coproduct has a recursive formula by a cocycle condition, the coproduct is made clear and useful by its explicit formula first in terms of admissible cuts and then in terms of subtrees and subforests. The recent work of Gao and Zhang~\mcite{ZG} on explicit construction of the coproduct in Loday-Ronco Hopf algebra of planar rooted trees is a similar contribution.

We are interested in the combinatorial construction of a Hopf algebra structures on free Rota-Baxter algebras by rooted trees.

The study of Rota-Baxter algebras originated from the work of G.Baxter~\mcite{Ba} on fluctuation theory in probability in 1960. It was studied by well-known mathematicians such as Atkinson, Cartier and Rota ~\mcite{FV,Ca,Ro} in the 1960-1970s. Its study has experienced a quite remarkable renascence in the recent decades with many applications in mathematics and physics ~\mcite{Ag,BBGN,EG,GK1,GK3,ML,MY,PBG,ZGG}, most notably the work of Connes and Kreimer on renormalization of quantum field theory ~\mcite{CK,CK1,EGK}. See ~\mcite{Gub} for further details and references.

As in the case of any algebraic structures, the understanding of free Rota-Baxter algebras is fundamental in the study of Rota-Baxter algebras and their applications. In the commutative case, the first construction of free commutative Rota-Baxter algebras by Rota~\mcite{Ro} led him to the close relationship between Spitzer's identity and Waring formula for symmetric functions. In the second construction of free commutative Rota-Baxter algebra~\mcite{Ca}, Cartier introduced a notion (stuffle) that became instrumental in the study of multiple zeta values~\mcite{BBBL} many years later. In the third construction~\mcite{GK1}, the authors gave a generalization of the shuffle product which turned out to be equivalent to the well-known quasi-shuffle product~\mcite{Ho} and stuffles. In the noncommutative case, free Rota-Baxter algebras have been constructed by various combinatorial objects, including bracketed words, leaf decorated forests and angularly decorated forests~\mcite{AM,EG,Guop}.
As in the commutative case, the different constructions of free Rota-Baxter algebras give different angles to study free Rota-Baxter algebras, even though they are naturally isomorphic.
More recently, a Hopf algebra structure has been given to free Rota-Baxter algebras on leaf decorated forests~\mcite{ZGG}.
The construction of free Rota-Baxter algebra is from a selected set of leaf decorated forests. The coproduct is obtained by first defining a coproduct on the whole space of leaf decorated forests and then taking the quotient to the space for the free Rota-Baxter algebra.
As such, the coproduct cannot be explicitly computed, since it is not clear how to obtain the coproduct of any given leaf decorated forest in the free Rota-Baxter algebra without taking quotients or going through a recursion.

In light of the importance of explicit constructions of the Connes-Kreimer and Loday-Ronco Hopf algebras mentioned above, for further study of the Hopf algebra on free Rota-Baxter algebras and for its applications, it is desirable to describe the coproduct directly on the rooted trees without the ambiguity of taking a quotient or the indirectness of going through a recursion.

This is the purpose of this paper. We will work with free Rota-Baxter algebras on angularly decorated planar rooted trees, following the construction in~\mcite{EG,Gub}. The advantage of this construction is that the underlying module is spanned by all planar rooted trees with angular decorations, in contrast to the construction by leaf decorated forests in~\mcite{ZGG} where a selected class of forests are used as representatives of the Hopf algebra on all leaf decorated planar rooted forests modulo the Rota-Baxter relation. We then introduce a coproduct on the angularly decorated planar rooted forests by suitably defining cuts and subforests, leading to a connected Hopf algebra structure on the free Rota-Baxter algebra.

The layout of the paper is as follows. In Section~\mref{sec:rbtree}, we first recall the notions of angularly decorated rooted forests and their use in constructing free Rota-Baxter algebras.
We then construct in Section~\mref{sec:bialg} a coproduct on free unitary Rota-Baxter algebra of angularly decorated rooted forests using a suitable notion of subforests, in analogy to the construction of the Connes-Kreimer coproduct on rooted trees. This coproduct is shown to be compatible with the multiplication on the free Rota-Baxter algebra, leading to a bialgebra structure on these forests. Finally the resulting bialgebra is shown in Section~\mref{ss:hopf} to be coaugmented, cofiltered and connected, hence can be enriched to a Hopf algebra.

\section{Rota-Baxter algebras and angularly decorated forests}
\mlabel{sec:rbtree}

In this section we recall the construction of the free Rota-Baxter algebra on angularly decorated forests.

\subsection{Angularly decorated forests}
\mlabel{ss:back}

We first recall the notion of Rota-Baxter algebras~\mcite{Ba,Gub}.
\begin{defn}
Let $\lambda$ be a given element of commutative ring $\bfk$. A {\bf Rota-Baxter algebra of weight $\lambda$} is a pair $(R, P)$ consisting of a $\bfk$-algebra $R$ and a linear operator $P:R\rightarrow R$ that satisfies the {\bf Rota-Baxter equation}
\begin{equation}
P(u)P(v)=P(uP(v))+P(P(u)v)+\lambda P(uv), \quad \forall u, v\in R.
\mlabel{ebr}
\end{equation}
\end{defn}

We give some basic examples of Rota-Baxter algebras and refer the reader to~\mcite{Gub} for more details.

\begin{exam} (Integration)
Let $R$ be the $\RR$-algebra of continuous functions on $\RR$. Define $P:R\to R$ by
the integration $$P(f)(x) = \int_0^x f(t) dt. $$
Then $P$ is a Rota-Baxter operator of weight 0.
\end{exam}

\begin{exam} (Scalar product)
Let $R$ be a \bfk-algebra. For any given $\lambda\in \bfk$, the operator
$$P_\lambda: R \to R\, \quad r\mapsto -\lambda r$$
is a Rota-Baxter operator of weight $\lambda$.
\end{exam}

\begin{exam} (Laurent series)
Let $R=\mathbb{C}[t^{-1},t]]$ be the algebra of Laurent series with coefficients in $\mathbb{C}$, consisting of series $\sum\limits_{n\geq N} a_n t^n$ where $N$ is any integer. Then the projection to the pole part:
$$ P\left(\sum_{n\geq N} a_n t^n\right):= \sum_{n<0} a_nt^n$$
is a Rota-Baxter operator of weight $-1$. Here the sum on the right is understood to be zero if $N\geq 0$. This operator plays an essential role in the study of renormalization of quantum field theory~\mcite{CK}. 
\end{exam}

In order to construct free Rota-Baxter algebras, we next recall the notions of planar rooted trees and planar rooted forests. Then we introduce angularly decorated rooted forests which will be our basic tools used in this paper.

A rooted tree is a connected and simply-connected set of vertices and oriented edges such that there is precisely one distinguished vertex, called the root, with no incoming edge. A planar rooted tree is a plane rooted tree with a fixed embedding into the plane. The following list shows the first few of them.

$$\ta1 \quad \tb2\quad\tc3\quad\td31\quad \te4\quad \thj44 \quad \tf41\quad\tg42\quad\th43$$

Let $\calt$ denote the set of planar rooted trees and $\calf$ the set of planar forests which can be identified with $S(\calt)$, the free semigroup generated by $\calt$ in which the product is denoted by $\bigsqcup$ or simply suppressed if there is danger of confusion. Then a planar rooted forest can be naturally expressed as an element of $S(\calt)$, of the form $T_1\bigsqcup T_2\cdots\bigsqcup T_n$ consisting of trees $T_1,\cdots, T_n.$ Here $\bigsqcup$ means putting two trees next to each other and will often be suppressed. Here some examples of planar rooted forests.

$$\ta1\sqcup \tc3 = \ta1\, \tc3,\quad
\tb2\sqcup \ta1\sqcup \tg42 = \tb2\, \ta1\, \tg42, \quad
\tc3\sqcup \td31\sqcup\tf41\sqcup\th43= \tc3\, \td31\,\tf41\,\th43$$

Obvious the multiplication $\bigsqcup$ satisfies the associativity.

We use $\lfloor T_1\bigsqcup T_2\cdots\bigsqcup T_n \rfloor$ to denote the tree obtained from the forest $T_1\bigsqcup T_2\cdots\bigsqcup T_n$ by adding a new root and an edge from the new root to each of the trees $T_1, \cdots, T_n$. In combinatorial terms, this is called the grafting of $T_1\bigsqcup T_2\cdots\bigsqcup T_n$ and is denoted by $B^+(T_1, \cdots ,T_n).$ So the operator $B^+$ is called the {\bf grafting operator}. For example,

$$ \lc \ta1\, \tb2\rc = B^+(\ta1\, \tb2) = \tg42, \quad \lc \ta1\,\ta1\,\ta1\rc = \th43.$$

For a rooted tree $T$, define the {\bf depth} $\dep(T)$ of $T$ to be the maximal length of the paths from the root to the leaves of the tree. For a forest
$F=T_1\bigsqcup T_2\cdots\bigsqcup T_\ell$ with rooted trees $T_1,T_2\cdots,T_\ell$, we define the {\bf depth} $\dep(F)$ of $F$ to be the maximum of the depths of the trees $T_1,\cdots,T_k$. We also define $\ell$ to be the {\bf length} of the forest $F$. So $\ell(F)$ is the number of tree factors in $F$.
For example,
  $$\ell(\di\,\tb2)=2, \ell(\tb2\,\di\, \tf41 )=3, \dep(\tb2)=1,\text{ }\dep(\tb2\, \tf41 )=2.$$

We now recall the construction of angularly decorated rooted trees. See~\mcite{EG2,Gub} for further details. 
  \begin{defn}
   Let X be a set.
   \begin{enumerate}
   \item
   An angularly decorated rooted tree is a planar rooted tree in which each angle (between two adjacent leafs) is decorated by an element of $X$.
\item
An angularly decorated rooted forest is a planar rooted forest with each angle is decorated by an element of $X$. Let $\calf_X^a$ denote the set of angularly decorated rooted forests with decoration set $X$. 
\end{enumerate}
\end{defn}
Note that the space between two rooted trees is taken as an angle. For example,
    $$ \di, \quad \tb2,\quad \xtd31, \quad \au4, \quad \ao3\quad $$
are angularly decorated trees; while
    $$ \quad \bullet x \bullet, \quad \di x_1\di x_2\di, \quad \tb2 x\di, \quad \di x_1\x2td31, \quad \tb2 x_1\x2td31$$
are angularly decorated forests.

From the definition of angularly decorated rooted forests, we add decorations to the angles of $\calf=S(\calt)$ to obtain angularly decorated forests. Intuitively, we use elements from $X$ to replace $\bigsqcup.$ So an angularly decorated rooted forest is of the form
\begin{equation}
T_1 x_1 T_2 x_2\cdots x_{\ell-1} T_\ell, \quad x_1,\cdots,x_{\ell-1}\in X,
\mlabel{eq:atdecomp}
\end{equation}
consisting of angularly decorated rooted trees $T_1,\cdots,T_\ell$.
The {\bf length} and {\bf depth} of an angularly decorated forest is defined to be the same as the underlying decorated forest. We note the the notion of length is different from the notion of breadth that we will introduced later.

\subsection{Rota-Baxter algebras by angularly decorated trees}
\mlabel{ss:treerba}

With notations in~Section~\mref{ss:back}, we let $\bfk\calf_X^a$ denote the free $\bfk$-module with basis $\calf_X^a$. We will equip $\bfk\calf_X^a$ with a Rota-Baxter algebra structure. In order to do this, we define a multiplication $\diamond_a$ on $\bfk\calf_X^a$.

For this purpose, we define
$$ \dia: \calf_X^a \times \calf_X^a \longrightarrow \bfk \calf_X^a$$
and then extend by bilinearity to a multiplication
$$ \dia: \bfk\,\calf_X^a \times \bfk\,\calf_X^a \longrightarrow \bfk\, \calf_X^a.$$
  The multiplication
$$ \dia: \calf_X^a \times \calf_X^a \longrightarrow \bfk \calf_X^a,$$
is defined recursively utilizing a grading structure on $\calf_X^a$ together with the grafting operator
$$B^+: \calf_X^a \longrightarrow \calf_X^a.$$

  The grading is given by the disjoint union (note the different meaning from the concatenation of trees)
$$ \calf_X^a = \bigsqcup_{n\geq 0} \calf_{X,n}^a,$$
where $\calf_{X,n}^a$ is the set of angularly decorated forests of depth $n$.
Then we have the linear grading
$$ \bfk\,\calf_X^a = \bigoplus_{n\geq 0} \bfk\,\calf_{X,n}^a.$$
  We will see later that the multiplication $\dia$ gives $\bfk\,\calf_X^a$ a filtered algebra, not a graded algebra. So we have to be careful.

  To be precise, the recursive definition of
$$ \dia: \calf_X^a \times \calf_X^a \longrightarrow \bfk \calf_X^a$$
means that we use induction on $n\geq 0$ to define

$$\dia{}_{,n}: \calf_{X,i}^a \times \calf_{X,j}^a \longrightarrow \bfk \calf_X^a$$
for all $i, j\geq 0$ with $i+j=n$.
Once this is achieved, then $\dia$ is well-defined as the direct sum of $\dia{}_{,n}, n\geq 0$, because of the disjoint union

$$ \calf_X^a \times \calf_X^a = \bigsqcup_{n\geq 0} \bigsqcup_{i+j=n} \calf_{X,i}^a \times \calf_{X,j}^a.$$

First let $n=0$. Then $i+j=n$ implies $i=j=0$. Note that
$$\calf_{X,0}^a = \left\{ \di x_1 \di \cdots \di x_k \di\,|\, k\geq 0\right \}$$
with the convention that $\di x_1 \di \cdots \di x_k \di=\di$ when $k=0$.
Then it is valid to define
 $$\dia{}_{,0}:\calf_{X,0}^a\times \calf_{X,0}^a\to \bfk\calf_X^a,$$
 by   $$\di \dia{}_{,0} \di=\di,\quad \di\dia{}_{,0} (\di x_1\di\cdots\di x_m\di)=(\di x_1\di\cdots\di x_m\di)\dia{}_{,0} \di=\di x_1\di\cdots\di x_m\di,$$
    $$(\di x_1\di\cdots\di x_m\di)\dia{}_{,0} (\di y_1\di\cdots\di y_n\di)=\di x_1\di\cdots\di x_m\di y_1\di\cdots \di y_n\di. $$

  For a given $k\geq 0$, assume that $\dia{}_{,m}, 0\leq m\leq k$, have been defined and we define $\dia{}_{,k+1}$. Then $k+1$ is greater or equal to 1.

  Let $T\in \calf_{X,i}^a,$ $T'\in \calf_{X,j}^a$ with $i+j=k+1\ge 1.$ We consider two cases.

(a) Suppose the length $\ell(T)=\ell(T')=1$, that is, $T$ and $T'$ are both angularly decorated rooted trees.
Then $T$ can be one and only one of the forms $\di$ or $B^+(\overline{T})$, and $T'$ can be one and only one of the forms $\di$ or $B^+(\overline{T}')$. Thus there are four cases and we define

\begin{equation}
T\dia{}_{\!\!\!,k+1} T' := \left\{ \begin{array}{ll}
\bullet,& \mathrm{\text{if } T=T'=\bullet},\\
T, & \mathrm{\text{if  }T'=\bullet},\\
T',& \mathrm{\text{if  }T=\bullet},\\
B^+(T\dia\overline{T'})+B^+(\overline{T}\dia T')+\lambda B^+(\overline{T}\dia\overline{T'}),& \mathrm{\text{if  }T=B^+(\overline{T}),T'=B^+(\overline{T'})}.
\end{array} \right.
\mlabel{eq:dia1}
\end{equation}
  Everything is clear except the last case. There we note that $\dep (B^+(\overline{T}))=\dep (\overline{T})+1$ and $\dep(B^+(\overline{T}'))=\dep (\overline{T}')+1$. So for the three terms in the last case, we have
  \begin{eqnarray*}
&\dep(T)+\dep(\ovf{T'})=&\dep(T)+\dep(T')-1=k;\\
&\dep(\ovf{T})+\dep(T')=&\dep(T)+\dep(T')-1=k;\\
&\dep(\ovf{T})+\dep(\ovf{T'})=&\dep(T)+\dep(T')-2=k-1.
\end{eqnarray*}
  Therefore $T\dia{}_{\!\!\!,k} \overline{T}'$, $\overline{T}\dia{}_{\!\!\!,k} T'$ and $\overline{T}\dia{}_{\!\!\!,k-1} \overline{T}'$ are all well-defined by the induction hypothesis.
  Thus the expression $T\dia{}_{\!\!\!,k+1} T'$ is well-defined.

(b) Suppose $\ell(T)=m\geq 2$ or $\ell(T')=n\geq 2$. Then $T=T_1 x_1\ldots x_{m-1} T_m$ and $T'=T_1'y_1\ldots y_{n-1} T_n'$ with $T_1,\cdots,T_m,\text{ }T_1',\cdots,T_n'\in \calf_X^a$, $x_1,\ldots ,x_{m-1},y_1,\ldots, y_{n-1}\in X$. Then we define:
\begin{equation}
 T\dia T':=T_1 x_1\ldots x_{m-2} T_{m-1} x_{m-1}(T_m\dia T_1')y_1T_2'y_2\ldots y_{n-1}T_n',
\mlabel{eq:prod2}
\end{equation}
with $T_m\dia T_1'$ defined in Case~(a). Note that $T_m\dia T_1'$ is a sum of angularly decorated trees by Case~(a). So the above equation gives a  well-defined sum of angularly decorated forests.

This completes the recursive definition of $\dia$ on $\calf_X^a$. Finally, as noted above, we expand the binary operation $\dia$ and $B^+$ to $\bfk \calf_X^a$ by bilinearity. Note that the multiplication $\dia$ is not commutative.
  For example, for
  $T_1:=\tb2,$ $T_2:=\di x\di$, we have
  $$T_1\dia T_2=\tb2 x\di\ne \di x\tb2 =T_2\dia T_1.$$

Adapting the arguments of~\mcite{EG,Gub}, we obtain
\begin{theorem}
\begin{enumerate}
\item
The triple $(\bfk\calf_X^a,\dia,B^+)$ is a non-commutative unitary Rota-Baxter algebra of weight $\lam$ with unit $\di$.
\item
Let $i_{x}:X\to \bfk\calf_X^a,\text{ }x\to \bullet x \bullet$ be the set map. The triple $(\bfk\calf_X^a,\dia,B^+,i_x)$ is a free non-commutative unitary Rota-Baxter algebra on a set $X$ characterized by the following universal property: for any  non-commutative unitary Rota-Baxter algebra $(R,\dirr,P)$ and any set map $f:X \to R$, there is a unique Rota-Baxter algebra homomorphism $\free{f}:\bfk\calf_X^a\to R$ such that $\free{f}\circ i_x=f.$
\end{enumerate}
\mlabel{thm:rba}
\end{theorem}

\section{Bialgebra structure on the free Rota-Baxter algebra}
\mlabel{sec:bialg}

In~\mcite{ZGG}, the free Rota-Baxter algebra on leaf decorated rooted forests was equipped with a bialgebra and Hopf algebra structure. Through the isomorphism between the free Rota-Baxter algebra in~\mcite{ZGG} and the free Rota-Baxter algebra on angularly decorated rooted forests in this paper, the bialgebra and Hopf algebra structures on the former free Rota-Baxter algebra can be transported to the latter one. However, in either case, the coproduct is defined by a recursion via a cocycle condition. Even though there is a combinatorial description of the coproduct on leaf decorated rooted forests, like in the work of Connes and Kreimer~\mcite{CK}, this combinatorial description does not carry over to the quotient which gives the free Rota-Baxter algebra on leaf decorated rooted trees. Thus such a definition of coproduct is not explicit and does not reveal possible relationship with the combinatorial properties of rooted forests.

In this section, we use combinatorial procedure to define a coproduct on the free Rota-Baxter algebra on angularly decorated rooted forests. This procedure is given in terms of substructures of angularly decorated rooted forest, in analogue to the substructures of leaf decorated rooted forests in the coproduct of Connes-Kreimer.

\subsection{Construction of the coproduct}
\mlabel{ss:copr}

First we define the counit
$$\epsilon_a:\bfk\calf_X^a\to \bfk$$
by sending $\bullet$ to $1_\bfk$ and $0$ otherwise. Also we denote $m:\bfk\calf_X^a\ot \bfk\calf_X^a\to \bfk\calf_X^a$ for the product $\dia$ defined in the last section and $u:\bfk\to \bfk\calf_X^a, 1_k\mapsto \di$ for the unit.

Now we give a combinatorial definition of a coproduct on angularly decorated rooted forests $\bfk \calf_X^a$.

Our construction is motivated by the coproduct of rooted trees and forests of Connes and Kreimer~\mcite{CK}, defined by subforests. So we briefly recall their definition.

Let $T$ be a rooted tree or forest.
Recall that a {\bf subtree} $T'$ of a tree $T$, denoted $T'\pq T$, is a vertex of $T$ together with its descendants and the edges connecting these vertices. A subtree is called {\bf nontrivial} if it is not the one vertex tree $\di$. More generally, a {\bf subforest} $F'$ of a forest $F=T_1\cdots T_k$, denoted $F'\pq F$, is $F'=T_1'\cdots T_k'$ where $T_i'\pq T_i, 1\leq i\leq k$. Equivalently, a subforest $F'$ of $F$ is a subset of vertices of $F$ together with the edges connecting them, so that if a vertex is in $F$, then all descendants of the vertex are in $F$.

In this language, the Connes-Kreimer coproduct of rooted forests is defined by

\begin{equation}
\Delta(F):= \sum_{F'\pq F} F' \ot (F/F'),
\mlabel{eq:tcopr1}
\end{equation}
where $F/F'$ is the forest obtained when the vertices of $F$ and edges (both internal and external) connecting to these vertices are removed from $F$.

Now let $F$ be an angularly decorated forest, with the decomposition
$$ F=T_1 x_1 T_2x_2\cdots x_{k-1}T_k$$
as in Eq.~(\mref{eq:atdecomp}). A vertex of $F$ is called a {\bf non-leaf vertex} if it is not a leaf.
A {\bf subtree} $T$ of $F$, denoted $T\pq F$, is a subset of vertices of $T$ together with the edges connecting them, so that if a vertex is in $F$, then all descendants of the vertex is in $F$.
The only vertex of $\di$ is regarded as a leaf.

\begin{defn}
Let $F$ be an angularly decorated forest. Let $\iu$ be a symbol not in $X$.
\begin{enumerate}
\item
A {\bf \rsubtree} of $F$ is a non-leaf subtree of $F$ as defined above for rooted trees together with all its angular decorations.
\item
A {\bf \lsubtree} is a set (in fact a vector) of decorations of a \rsubtree without the underlying subtree.
\item
A {\bf \vsubtree} is either a \rsubtrees or a \lsubtrees.
\item
A {\bf (angularly decorated) \vsubforest} $H$ of an angularly decorated forest $F$, denoted $H\pq F$, consists of a sequence $H_1,\cdots, H_n$ of mutually disjoint \vsubtrees of $F$ in the order that they appear in $F$.
\item
The {\bf \closure} of a \vsubforest $H=H_1\cdots H_n$ of $F$, denoted $\cl(H)$, is the angularly decorated forest obtained from expanding $H$ as follows.
\begin{enumerate}
\item
If $H_1$ (resp. $H_n$) is a \lsubtree, then replace $H_1$ (resp. $H_n$) by $\di H_1$ (resp. $H_n \di$);
\item
If $H_i$ and $H_{i+1}$ are both \lsubtrees, then replace $H_iH_{i+1}$ by $H_i \di H_{i+1}$
\item
If $H_i$ and $H_{i+1}$ are both \rsubtrees, then replace $H_iH_{i+1}$ by $H_i \iu H_{i+1}$. \end{enumerate}
The role of the symbol $\iu$ is to represent the operation $\dia$ without executing it, that is, without multiplying out $H_i\dia H_{i+1}$, in order to keep and show the combinatorial structure.

This closure is called the {\bf angularly decorated forest generated by $H$}.
\item
For a \vsubforest $H=H_1\cdots H_n$ of $F$, the {\bf quotient forest} $F/H$ is obtained from $F$ by carrying out the following procedure for each $H_i, 1\leq i\leq n$:
\begin{enumerate}
\item
if $H_i$ is a \rsubtree, then take out of $F$ as for the usual rooted forests;
\item
if $H_i=\{x\}$ is a \lsubtree, then replace $x$ by $\iu$ at the angle that $x$ decorates.
\end{enumerate}
\end{enumerate}
\end{defn}

In both case, the role of $\iu$ is the multiplication $\dia$. See the examples below.

\begin{exam} \mlabel{ex:adsub}
\begin{enumerate}
\item
For the angularly decorated tree $\xtd31$, the \rsubtrees are the trivial tree $\bullet$ and $\xtd31$, the only \lsubtree is $x$. Thus the \vsubforests are
$\bullet, \xtd31, x$. Their closures are
\begin{equation}\mlabel{eq:subfxtd31}
\bullet, \xtd31, \bullet x \bullet.
\end{equation}
The corresponding quotients are
$$ \xtd31, \bullet,  \ap1.$$
\item
For the angularly decorated tree $\ay5$, the \rsubtrees are
$ \bullet, \ax10,\  \ay5$, the \lsubtrees are $x_1, x_2$. Thus the \vsubforests, their closures and quotients are
\smallskip

\begin{tabular}{|c||c|c|c|c|c|c|c|}
\hline
\vsubforests &
 $\bullet$& $x_1$& $\ax10$& $x_2$& $x_1 x_2$& $\ax10 x_2$& $\ay5$ \\
 &&&&&&&\\ \hline
closures
& $\bullet$ &$\bullet x_1\bullet$& $\ax10$ & $\bullet x_2\bullet$ & $\bullet x_1\bullet x_2\bullet$ & $\ax10 x_2\bullet$ & $\ay5$\\
&&&&&&&\\ \hline
quotients & $\ay5$ & $ \az6$ & $\av9$ & $\as7$ & $\aw8$ & $\ap1$ & $\di$ \\
&&&&&&& \\ \hline
\end{tabular}
\end{enumerate}
\end{exam}

  \begin{defn}
   For $F\in \bfk\calf_X^a$, with notations above, we define the {\bf angular coproduct} of $F$ by
\begin{equation}
\sj(F):=\sum_{G\pq T} \cl(G)\ot T/G.
\mlabel{eq:acopr3}
\end{equation}
If $\iu$ appears in the right hand side, then replace $\iu$ by $\dia$.
   \end{defn}

\begin{exam} \mlabel{ex:adcoprod}
For the angularly decorated trees in Example~\mref{ex:adsub}, we have
\begin{eqnarray*}
   &\sj(\xtd31)&=\bullet\ot \xtd31 +\xtd31\ot \bullet+ \bullet x\bullet\ot \ap1 \quad \text{(by Eq.~\eqref{eq:acopr3})}\\
   &&=\bullet\ot\xtd31+ \xtd31\ot \bullet+\bullet x\bullet\ot B^+(\bullet\iu \bullet) \quad \text{(by the definition of $B^+$)}\\
   &&=\bullet\ot\xtd31+ \xtd31\ot \bullet+\bullet x\bullet\ot \tb2 \quad \text{(by $\bullet\iu \bullet:=\bullet \dia \bullet=\bullet$)}
   \end{eqnarray*}
For another example, we compute
  \begin{eqnarray*}
   \sj(\ay5)&=&\di\ot \ay5+\di x_1\di\ot \az6+\ax10\ot\av9+\di x_2\di\ot\as7\\
   &&+\di x_1\di x_2\di\ot \aw8+\ax10 x_2\di\ot\ap1+\ay5\ot \di \quad \text{(by Eq.~\eqref{eq:acopr3})}\\
   &=&\di\ot \ay5+\di x_1\di\ot B^+(B^+(\di\iu\di)x_2\di)+\di x_2\di\ot B^+(\ax10\iu\di)\\
   &&+\di x_1\di x_2\di\ot B^+(B^+(\di\iu \di)\iu\di)+\ax10 x_2\di\ot B^+(\di\iu\di)+\ax10\ot\av9+\ay5\ot \di\\
   && \hspace{7cm} \text{(by the definition of $B^+$)}\\
   &=&\di\ot \ay5+\di x_1\di\ot B^+(\tb2 x_2\di)+\di x_2\di\ot B^+(\ax10)\\
   &&+\di x_1\di x_2\di\ot B^+(\tb2)+\ax10 x_2\di\ot B^+(\di)+\ax10\ot\av9 +\ay5\ot \di\\
   && \hspace{7cm} \text{(by $\bullet\iu \bullet=\bullet \dia \bullet =\bullet$)}\\
   &=&\di\ot \ay5+\di x_1\di\ot\au4+\di x_2\di\ot\ao3\\
   &&+\di x_1\di x_2\di\ot \tc3+\ax10 x_2\di\ot \tb2+\ax10\ot\av9 +\ay5\ot \di\\
   && \hspace{7cm}  \text{(by the definition of $B^+$)}\\
   \end{eqnarray*}
\end{exam}

   \smallskip

Next we give another description of the angular coproduct in compatible with the decomposition
of $F$ in~(\mref{eq:atdecomp}):
$$ F=T_1x_1T_2x_2\cdots x_{k-1}T_k,$$
for $x_1,\cdots x_{k-1}\in X,$ $T_1,\cdots T_k\in \calf_X^a.$
Then a \vsubforest of $F$ is of the form, called a {\bf factorwise \vsubforest}
$$ F'=T_1' x_1' T_2' x_2' \cdots x_{k-1}'T_k',$$
where $T_i'$ is a \vsubforest of $T_i, 1\leq i\leq k,$ and $x_i'$ is either $x_i$ or $\iu, 1\leq i\leq k$. Then the quotient forest $F/F'$ is
$$ F/F'=(T_1/T_1')(x_1/x_1')(T_2/T_2')(x_2/x_2')\cdots (x_{k-1}/x_{k-1}')(T_k/T_k'),$$
where $T_i/T_i'$ is the quotient tree and $x_i/x_i'$ is $\iu$ or $x_i$ depending on $x_i'$ being $x_i$ or $\iu$.  It is called a {\bf factorwise quotient forest}.

Then we have the following alternative definition of $\sj$.
For $F\in\calf_X^a$ with the decomposition
 $F=T_1 x_1 T_2 x_2 \cdots x_{k-1} T_k$ where  $T_1,$ $\cdots,$ $T_k\in \calf_X^a,$ and $x_1,$ $\cdots,$ $x_{k-1}\in X.$
 Then with the notions above, we have
\begin{equation}
\sj(F):=\sum_{F'\pq F}F'\ot F/F'.
\mlabel{eq:acopr4}
\end{equation}

This description is particularly convenient when there are multiple tree factors in a forest, as shown in the following example.

\begin{exam} \mlabel{ex:coprod2}
Consider the angularly decorated forest $F=\ax10 x_2\bullet$. The following table gives the \vsubforests, their closures and the \vsubforests given factor-by-factor. We first note that
$\ax10$ has three \vsubforests, $\bullet$ has one and $x_2'$ has two two choices: $x_2'=\iu, x_2$. Thus altogether, there are six \vsubforests of $F$. Their corresponding closures, factorwise subforests, factorwise quotients are listed in the following table. \\

\begin{tabular}{|c||c|c|c|c|c|c|}
\hline
\vsubforests & $\bullet$ & $x_1$ & $\ax10$ & $x_2$ & $x_1 x_2$ & $\ax10 x_2 $\\ \hline
closures & $\bullet$ 	& $\bullet x_1\bullet $ & $\ax10$ & $\bullet x_2\bullet$ & $\bullet x_1\bullet x_2\bullet $ & $\ax10 x_2 \bullet$\\ \hline
factorwise \vsubforests & $\bullet \iu \bullet \iu \bullet$ 	& $\bullet x_1\bullet\iu \bullet $ & $\ax10\iu \bullet$ & $\bullet \iu \bullet x_2\bullet$ & $\bullet x_1\bullet x_2\bullet $ & $\ax10 x_2 \bullet$\\ \hline
factorwise quotients & $\ax10 x_2 \bullet$ & $\ap1 x_2 \bullet$ &  $\bullet \iu \bullet x_2 \bullet$ & $\ax10 \iu \bullet $ & $\ap1 \iu\bullet$ & $\bullet\iu \bullet \iu \bullet$ \\ \hline
reduced quotients & $\ax10 x_2 \bullet$ & $\ap1 x_2 \bullet$ &  $ \bullet x_2 \bullet$ & $\ax10 $ & $\ap1 $ & $\bullet$ \\ \hline
\end{tabular}
\smallskip
\\

Thus we have the coproduct
  \begin{eqnarray*}
  &\sj(\ax10 x_2\bullet)
  &= \bullet\ot \ax10 x_2\bullet+\bullet x_1\bullet\ot \ap1 x_2\bullet+\ax10\ot \bullet x_2\bullet+\bullet x_2\bullet\ot\ax10\\
  &&+(\bullet x_1\bullet x_2\bullet)\ot \ap1 + \ax10 x_2\bullet\ot \bullet\\
  &&= \bullet\ot \ax10 x_2\bullet+\bullet x_1\bullet\ot B^+(\bullet\iu\bullet)x_2\bullet+\ax10\ot \bullet x_2\bullet+\bullet x_2\bullet\ot \ax10\\
  &&+(\bullet x_1\bullet x_2\bullet)\ot B^+(\di \iu \di) + \ax10 x_2\bullet\ot \bullet\\
  &&= \bullet\ot \ax10 x_2\bullet+\bullet x_1\bullet\ot\tb2 x_2\bullet+\ax10\ot \bullet x_2\bullet+\bullet x_2\bullet\ot\ax10\\
  &&+\bullet x_1\bullet x_2\bullet\ot\tb2+ \ax10 x_2\bullet\ot \bullet
  \end{eqnarray*}
\end{exam}

\subsection{The bialgebra structure}
\mlabel{ss:prop}
  \begin{theorem}
  Let $\sj:\bfk\calf_X^a\to \bfk\calf_X^a\ot \bfk\calf_X^a$ be the angular coproduct defined in Eq.~(\mref{eq:acopr3}) or (\mref{eq:acopr4}). Then $\sj$ satisfies the follow properties.
  \begin{enumerate}
  \item $\sj(\bullet)=\bullet\ot \bullet;$
  \mlabel{it:acopr1}
  \item $\sj(\di x\di)=\di x\di\ot\di+\di\ot\di x\di,$ $x\in X$;
  \mlabel{it:acopr2}
  \item
  $\sj(B^+(F))=B^+(F)\ot\di+(\id\ot B^+)(\sj(F))$ for all $F\in \calf_X^a$;
  \mlabel{it:acopr3}
  \item
  $\sj(F_1\dia F_2)=\sj(F_1)\dia\sj(F_2)$ for $F_1,$ $F_2\in\calf_X^a$.
  \mlabel{it:acopr4}
  \end{enumerate}
  \mlabel{thm:acopr}
  \end{theorem}
  \begin{proof}
  By the definition of $\sj$, it is direct that Items~(\mref{it:acopr1}) and (\mref{it:acopr2}) hold.
  \smallskip

(\mref{it:acopr3}) We verify
   $$\sj(B^+(F))=B^+(F)\ot\di+(\id\ot B^+)(\sj(F))\text{  for all }F\in \bfk\calf_X^a$$
by the same argument for the cocycle property of the Connes-Kreimer coproduct for rooted trees, here made possible by the combinatorial description of the angular coproduct. Consider the coproduct
$$ \sj(B^+(F))=\sum_{G\pq B^+(F)} \cl(G) \ot F/G.$$
If $G\pq B^+(F)$ contains the root of $B^+(F)$, then $G=B^+(F)$ and the corresponding term in the sum is $B^+(F)\ot \di$. If $G\pq B^+(F)$ does not contain the root of $B^+(F)$, then by the definition of angular subforests, we have $G\pq F$. Further the corresponding quotient forest is obtained from the grafting of $F/G$. Therefore the corresponding term in the sum is $G \ot B^+(F/G)$. In summary, we obtain

\begin{eqnarray*}
\sj(B^+(F))&=& B^+(F)\ot \di + \sum_{G\pq F} \cl(G)\ot B^+(F/G)\\
&=& B^+(F)\ot \di + (\id \ot B^+) \left( \sum_{G\pq F} \cl(G)\ot F/G\right) \\
&=& B^+(F)\ot \di + (\id \ot B^+) \sj(F),
\end{eqnarray*}
as needed.

\smallskip
  \noindent
  (\mref{it:acopr4}). 
\delete{
By Item~(\mref{it:acopr3}), the coproduct $\sj$ satisfies the same cocycle condition as the coproduct on the free Rota-Baxter algebra on rooted trees in~\mcite{ZGG}. Further by Items~(\mref{it:acopr1}) and (\mref{it:acopr2}), the two coproducts also satisfy the same initial conditions.
  Further, the Rota-Baxter algebra structures are also identified. Therefore, the bialgebra and Hopf algebra structure on free Rota-Baxter algebra on leaf decorated forests transports to the same structures on the free Rota-Baxter algebra on angularly decorated forests discussed here. In particular, Item~(\mref{it:acopr4}) holds. For completeness, we give a direct proof of the multiplicativity of $\sj$ below.

We give a direct proof of 
}
We prove the desired multiplicativity by induction on the sum $\dep(F_1)+\dep(F_2)$ of depths of $F_1$ and $F_2$ in $\calf_X^a$.

First when $\dep(F_1)+\dep(F_2)=0$, then $F_1=\di x_1\di \cdots \di x_m$ and $F_2=\di x_{m+1}\di \cdots \di x_{m+n}$ for some $m, n\geq 1$.
For a set $I=\{i_1<\cdots <i_r\}$ of positive integers, we use the notation
$\di x_I \di: =\di x_{i_1}\di \cdots \di x_{i_r}\di$. Also denote $[m]=\{1, \cdots, m\}$ and $[m+1,m+n]:=\{m+1, \cdots, m+n\}$. Then for $I\subseteq [m]$ we have $(\di x_{[m]} \di)/(\di x_{[m]\backslash I} \di)$.
With these notations,  we obtain
$$\sj(F_1)=\sum_{I\subseteq [n]} (\di x_I \di) \ot  (\di x_{[m]\backslash I}\di), \sj(F_2)=\sum_{J\subseteq [m+1,m+n]} (\di x_J \di) \ot (\di x_{[m+1,m+n]\backslash J}\di).$$
Therefore,
\begin{eqnarray*}
\sj(F_1) \dia \sj(F_2)&=&
\sum_{I\subseteq [m], J\subseteq [m+1,m+n]}
\Big((\di x_I\di )\dia (\di x_J\di)\Big) \ot \Big( \big(\di x_{[m]\backslash I} \di\big)\dia \big(\di x_{[m+1,m+n]\backslash J}\di\big)\Big) \\
&=& \sum_{L\subseteq [m+n]} (\di x_L \di ) \ot (\di x_{[m+n]\backslash L} \di)\\
&=& \sj(F_1\dia F_2).
\end{eqnarray*}

Next assume that for $k\geq 0$, Item~(\mref{it:acopr3}) holds whenever $\dep(F_1)+\dep(F_2)\leq k$. Consider $F_1, F_2\in \calf_X^a$ with $\dep(F_1)+\dep(F_2)=k+1$. We first consider the case when the breadths of $F_1$ and $F_2$ are one. In this case, if further one of $F_1$ or $F_2$ has depth zero and so is of the form $\di x\di$, then by the definition of $\dia$ in Eqs.~(\mref{eq:dia1}) and (\mref{eq:prod2}), we have
$F_1\dia F_2 = \di x F_2$ or $F_1\dia F_2= F_1 x \di$. Then it is direct to check that Item~(\mref{it:acopr3}) holds.
In the remaining case when $F_1$ and $F_2$ both have positive depths, then $F_1=B^+(\ovf{F_1})$ and $F_2=B^+(\ovf{F_2})$.
Denote
$$\ovf{F_1}\star \ovf{F_2}:= F_1\dia \ovf{F_2}+\ovf{F_1}\dia F_2 +\lambda \ovf{F_1}\dia \ovf{F_2},$$
so that $F_1\dia F_2=B^+(\ovf{F_1} \star \ovf{F_2})$. Then by the cocycle condition and the induction hypothesis, we have
\begin{eqnarray*}
 \sj(F_1\dia F_2)&=&\sj(B^+(\ovf{F_1}\star \ovf{F_2})) \\
&=& B^+(\ovf{F_1}\star \ovf{F_2})\ot \di + (\id \ot B^+)(\sj(\ovf{F_1}\star \ovf{F_2}))\\
&=& B^+(\ovf{F_1}\star \ovf{F_2}) \ot \di + (\id \ot B^+)(\sj(\ovf{F_1})\star \sj(\ovf{F_2})).
\end{eqnarray*}

\begin{eqnarray*}
\sj(F_1\dia F_2)&=& \sj\Big(B^+(F_1\dia \ovf{F_2})+B^+(\ovf{F_1}\dia F_2) + \lambda B^+(\ovf{F_1}\dia \ovf{F_2})\Big) \\
&=& B^+(F_1\dia \ovf{F_2})\ot \di + (\id \ot B^+)(\sj(F_1\dia \ovf{F_2})) + B^+(\ovf{F_1}\dia F_2)\ot \di \\
&&+ (\id \ot B^+)(\sj(\ovf{F_1}\dia F_2)) + \lambda B^+(\ovf{F_1}\dia \ovf{F_2})\ot \di
+ \lambda (\id \ot B^+)(\sj(\ovf{F_1}\dia \ovf{F_2}))\\
&=& (F_1\dia F_2) \ot \di + (\id \ot B^+)(\sj(F_1)\dia \sj(\ovf{F_2})) \\
&&+ (\id \ot B^+)(\sj(\ovf{F_1})\dia \sj( F_2))
+ \lambda (\id \ot B^+)(\sj(\ovf{F_1})\dia \sj(\ovf{F_2}))\\
&=& (F_1\dia F_2)\ot \di + (\id \ot B^+)\big((F_1\ot \di + (\id\ot B^+)(\sj(\ovf{F_1}))\dia \sj(\ovf{F_2})\big) \\
&& + (\id \ot B^+)\Big(\sj(\ovf{F_1})\dia \big(F_2\ot \di + (\id\ot B^+)(\sj(\ovf{F_2}))\big)\Big)\\
&&+ \lambda (\id \ot B^+)(\sj(\ovf{F_1})\dia \sj(\ovf{F_2}))\\
&=& (F_1\dia F_2)\ot \di + (\id \ot B^+)\Big((F_1\ot \di)\dia \sj(\ovf{F_2})\Big) \\
&&+ (\id \ot B^+)\Big((\id\ot B^+)(\sj(\ovf{F_1}))\dia \sj(\ovf{F_2})\Big) \\
&& + (\id \ot B^+)\Big(\sj(\ovf{F_1})\dia \big(F_2\ot \di)\Big) + (\id\ot B^+)\Big(\sj(\ovf{F_1})\dia \big((\id\ot B^+)(\sj(\ovf{F_2}))\big)\Big)\\
&&+ \lambda (\id \ot B^+)(\sj(\ovf{F_1})\dia \sj(\ovf{F_2}))
\end{eqnarray*}

It is a general fact that if $P$ is a Rota-Baxter operator on an algebra $R$, then $\id \ot P$ is a Rota-Baxter algebra on the tensor product algebra $R\ot R$. Thus combining the third, fourth and fifth terms of the above equation gives $(\id\ot B^+)(\sj(\ovf{F_1})) \dia (\id\ot B^+)(\sj(\ovf{F_2}))$.
Thus from the above equation we obtain
\begin{eqnarray*}
\sj(F_1\dia F_2) &=& (F_1\dia F_2)\ot \di + (\id \ot B^+)\Big((F_1\ot \di)\dia \sj(\ovf{F_2})\Big) \\
&&+  (\id \ot B^+)\Big(\sj(\ovf{F_1})\dia \big(F_2\ot \di)\Big) + (\id\ot B^+)(\sj(\ovf{F_1})) \dia (\id\ot B^+)(\sj(\ovf{F_2})).
\end{eqnarray*}

On the other hand, we have
{\small
\begin{eqnarray*}
\sj(F_1)\dia \sj(F_2)&=&
\Big(B^+(\ovf{F_1})\ot \di +(\id\ot B^+)(\sj(\ovf{F_1}))\Big)\dia \Big(B^+(\ovf{F_2})\ot \di +(\id\ot B^+)(\sj(\ovf{F_2}))\Big) \\
&=& (F_1\dia F_2)\ot \di + (F_1\ot \di)\dia \big((\id\ot B^+)(\sj(\ovf{F_2}))\big) \\
&&+ \big((\id\ot B^+)(\sj(\ovf{F_1}))\big) \dia (F_2\ot \di)
+ \big((\id\ot B^+)(\sj(\ovf{F_1}))\big) \dia \big((\id\ot B^+)(\sj(\ovf{F_2}))\big).
\end{eqnarray*}
}

Since
$$ (\id \ot B^+)\Big((F_1\ot \di)\dia \sj(\ovf{F_2})\Big) =
(F_1\ot \di)\dia \big((\id\ot B^+)(\sj(\ovf{F_2}))\big),$$
we find that
$$\sj(F_1\dia F_2)=\sj(F_1)\dia \sj(F_2).$$

Finally when $F_1$ and $F_2$ have breadths $r\geq 1$ and $s\geq 1$ respectively,  with standard decompositions
$$F_1=F_{1,1}\dia \cdots \dia F_{1,r}, \quad F_2=F_{2,1}\dia \cdots \dia F_{2,s}.$$
Then noting that $\sj$ is defined to be compatible with the standard decomposition (see the alternative description) and that the standard decomposition of $F_1\dia F_2$ is
$$ F_1\dia F_2 = F_{1,1}\dia \cdots \dia F_{1,r-1}\dia (F_{1,r}\dia F_{2,1}) \dia F_{2,2} \dia \cdots \dia F_{2,s}.$$
Thus applying the previous case, we obtain
\begin{eqnarray*}
\sj(F_1\dia F_2)&=& \sj(F_{1,1}\dia \cdots \dia F_{1,r-1}\dia (F_{1,r}\dia F_{2,1}) \dia F_{2,2} \dia \cdots \dia F_{2,s}) \\
&=& \sj(F_{1,1})\dia \cdots \dia \sj(F_{1,r-1})\dia \sj(F_{1,r}\dia F_{2,1}) \dia \sj(F_{2,2}) \dia \cdots \dia \sj(F_{2,s}) \\
&=& \sj(F_{1,1})\dia \cdots \dia \sj(F_{1,r-1})\dia \sj(F_{1,r})\dia \sj(F_{2,1}) \dia \sj(F_{2,2}) \dia \cdots \dia \sj(F_{2,s}) \\
&=& \sj(F_{1,1}\dia \cdots \dia \sj(F_{1,r}))\dia \sj(F_{2,1} \dia  \cdots \dia F_{2,s}).
\end{eqnarray*}
This completes the proof of Item~(\mref{it:acopr4}).
\end{proof}

Now we verify the other conditions for $\bfk\calf_X^a$ to be a bialgebra.

\begin{theorem}
The quintuple $(\bfk\calf_X^a,m,u,\sj,\epsilon_a)$ is a bialgebra. 
\delete{
More precisely, $\sj$ and $\epsilon_a$ are algebra homomorphisms and the following diagrams commute.

\begin{equation}
\begin{split}
\xymatrix{
\bfk\calf_X^a \ar[r]^{\sj} \ar[d]_{\sj} &\bfk\calf_X^a\ot \bfk\calf_X^a \ar[d]^{\id\ot \sj}\\
\bfk\calf_X^a\ot \bfk\calf_X^a \ar[r]_<<<<{\sj\ot \id} & \bfk\calf_X^a\ot \bfk\calf_X^a\ot \bfk\calf_X^a}
\end{split}
\mlabel{eq:coasso}
\end{equation}

\begin{equation}
\begin{split}
\xymatrix{
 &\bfk\calf_X^a \ar[ld]_{\beta_l} \ar[d]^{\sj} \ar[rd]^{\beta_r}\\
\bfk\ot \bfk\calf_X^a &\bfk\calf_X^a\ot \bfk\calf_X^a \ar[l]_{\eb \ot \id} \ar[r]^<<<{\id\ot \eb} & \bfk\calf_X^a\ot \bfk }
\mlabel{eq:counit}
\end{split}\end{equation}
where $\beta_l:\bfk\calf_X^a\to \bfk\ot \bfk\calf_X^a,$ $F\to 1_k\ot F$, $\beta_r:\bfk\calf_X^a\to \bfk\calf_X^a\ot \bfk,$ $F\to F\ot 1_k.$
\begin{equation}
\begin{split}\xymatrix{
 \bfk\ar[d]_{\vartriangle_k} \ar[r]^u &\bfk\calf_X^a \ar[d]^{\sj}\\
\bfk\ot \bfk \ar[r]_<<<<{u\ot u} & \bfk\calf_X^a\ot \bfk\calf_X^a}
\qquad	\xymatrix{
\bfk\ar[r]^u \ar[rd]_{\id_\bfk} &
\bfk\calf_X^a \ar[d]^\eb\\
& \bfk}
\end{split}
\end{equation}
}
\mlabel{thm:abial}
\end{theorem}

\begin{proof}
By Theorem~\mref{thm:acopr}, the natural algebraic isomorphism between $(\bfk\calf_X^a,m,u)$ and the free Rota-Baxter algebra on $X$ in~\mcite{ZGG} preserves the coproducts. Then since the coproduct in~\mcite{ZGG} is compatible with the product and gives rise to a bialgebra, the same holds for the quintuple $(\bfk\calf_X^a,m,u,\sj,\epsilon_a)$. 
\end{proof}

\section{The Hopf algebra structure}
\mlabel{ss:hopf}

We end the paper by showing that the bialgebra of angularly decorated forests obtained in the last section is a Hopf algebra.

\begin{defn}
  A {\bf coaugmented coalgebra} is a quadruple $(C, \triangle, \epsilon,u)$ where $(C,\triangle,\epsilon)$ is a coalgebra and  $u:k\to C$ is a linear map,
called the coaugmentation, such that $\epsilon \circ u= \id_k.$
\end{defn}

In Section~\mref{sec:bialg}, we have defined $\epsilon_a:\bfk\calf_X^a\to \bfk$ and $u:\bfk\to \bfk\calf_X^a$, so that
  $$\eb \circ u=\id_k$$
  In other words, we have shown that $(\bfk\calf_X^a,\sj,\eb)$ is a coaugmented coalgebra.
\begin{defn}
(\mcite{GG}) A bialgebra $(H,m,u,\Delta,\epsilon)$ is called {\bf cofiltered} if there are $\bfk-$submodules $H^n$, $n\ge0$, such that
\begin{enumerate}
	\item
 $H^n\subseteq H^{n+1}$ for all $n\ge 0$;
 \item $H=\cup_{n\ge0}^\infty H^n$ for all $n\ge 0$;
 \item $\Delta(H^n)\subseteq \sum_{p+q=n} H^p\ot H^q,$ $n\ge0$;
 \item $H^n=\im u\oplus (H^n\cap \ker\epsilon),$
where $p,q\ge 0$.
$H$ is called connected (cofiltered) if in addition $H^0=\im u$.
\end{enumerate}
\end{defn}

\begin{defn}
 Let $\deg(F)$ denote the number of vertices of $F\in \calf_X^a.$
\end{defn}
Then we have $\deg(F_1\dia F_2)=\deg(F_1)+\deg(F_2)-1.$

Now we prove
\begin{prop}
With the above notations, $\bfk\calf_X^a$ is a connected, cofiltered coaugmented coalgebra.
\mlabel{pp:conn}
\end{prop}
\begin{proof}
 First we define
$$\fraka^n:=\{F\in \calf_X^a|\deg(F)-1\le n\} \text{ for } n\ge0.$$
  And we denote $H^n:=\bfk\fraka^n.$ Then we have $H^0= \bfk=\im u$ and $\bfk\calf_X^a=\cup_{n\ge0}^\infty H^n$, so $(b) $ is clear.

(a) Obviously, for $F\in H^n,\text{ }\deg(F)\le n+1,$
$$H^{n+1}:=\bfk\fraka^{n+1}:=\bfk\{F\in \calf_X^a| \deg(F)-1\le n+1\}.$$
So, $F\in H^{n+1},$ that is, $H^n\subseteq H^{n+1}.$

(c) When $n=0,$ $F=\bullet\in H^0,$
$$\sj(F)=\sj(\bullet)=\bullet\ot \bullet\subseteq \sum_{0+0=0} H^0\ot H^0.$$
Assume that for $n=k\geq 0,\text{ }\sj(H^k)\subseteq \sum_{p+q=k} H^p\ot H^q$. Then we consider the case of $n=k+1$. Let $F\in H^n$. We consider two cases.

{\bf Case 1.} If $\bre(F)=1$, then we have $F=B^+(\ovf{F})\text{ and }\deg(\ovf{F})=k+1.$ Also we have
$$\sj(F)=\sj(B^+(\ovf{F}))=F\ot \bullet+(\id\ot B^+ )\circ \sj(\ovf{F}),$$
so we have $(\id\ot B^+ )\circ \sj(\ovf{F})\in \sum_{p+q+1=k+1}H^p\ot H^{q+1}$ and $F\ot \bullet\in H^{k+1}\ot H^0$. Then
 $$\sj(F)=\sj(H^{k+1})\subseteq \sum_{p+q+1=k+1}H^p\ot H^q. $$

{\bf Case 2.} If $\bre(F)\ge 2,$ we have $F=F_1\dia F_2,\text{ and }\deg(F)=\deg(F_1)+\deg(F_2)-1=k+2$, where we denote
$$\sj(F_1)\in \sum_{p_1+ q_1=\deg(F_1)-1}H^{p_1}\ot H^{q_1}, \quad \sj(F_2)\in\sum_{p_2+ q_2=\deg(F_2)-1}H^{p_2}\ot H^{q_2}.$$
\begin{eqnarray*}
&\sj(F)&=\sj(F_1)\dia \sj(F_2)\\
&&\in \left(\sum_{p_1+ q_1=\deg(F_1)-1}H^{p_1}\ot H^{q_1}\right)\dia \left(\sum_{p_2+ q_2=\deg(F_2)-1}H^{p_2}\ot H^{q_2}\right)\\
&&=\sum_{p_1+ q_1=\deg(F_1)-1}\sum_{p_2+ q_2=\deg(F_2)-1}(H^{p_1}\dia H^{p_2})\ot (H^{q_1}\dia H^{q_2})\\
&&\subseteq \sum_{p_1+ q_1=\deg(F_1)-1}\sum_{p_2+ q_2=\deg(F_2)-1}H^{p_1+p_2}\ot H^{q_1+q_2}\\
&&\subseteq \sum_{p+q=k+1}H^p\ot H^q.
\end{eqnarray*}
This completes the induction.
\smallskip

\noindent
(d) Since $\eb \circ u=\id_\bfk$, $u\circ \eb$ is idempotent. Further $u$ is injective and $\eb$ is surjective. Thus
$$H^n = \im u\circ \eb|_{H^n} \oplus \ker u\circ \eb|_{H^n} = \im u\oplus (H^n \cap\ker\eb).$$

In summary, we have proved that $(\bfk \calf_X^a,\sj,\eb)$ is a connected coaugmented cofiltered coalgebra.
\end{proof}

\begin{lemma} \mlabel{lem:conn}
~\mcite{GG} Let $(H,m,u,\sj,\varepsilon)$ be a bialgebra such that $ (H,\sj,\varepsilon,u)$ is a connected coaugmented cofiltered coalgebra is a Hopf algebra and the antipode $S$ is given by \\
$$S(1_H)=1_H\text{ and } S(x)=-x+\sum_{n\geqslant 1}(-1)^{n+1}m^n\bar{\Delta}^n(x) \text{ for }x\in \ker\varepsilon,$$
where $\bar{\Delta}(x):=\Delta(x)-1_H\ot x- x\ot 1_H\in \ker \varepsilon \otimes \ker \varepsilon.$ 
\end{lemma}
Here are examples of the antipodes for some angularly decorated forests.
$$ S(\xtd31)=-\xtd31 +\ta1\, x\, \tb2,$$
$$ S(\xtd31 y\, \ta1)=\ta1\, y \xtd31 -\ta1\, y\, \ta1\, x\, \tb2.$$
 
Combining Proposition~\mref{pp:conn} and Lemma~\mref{lem:conn}, we obtain
  \begin{theorem}
 $(\bfk \calf_X^a,\dia, u, \sj,\eb,S)$ is a Hopf algebra.
\end{theorem}

\smallskip

\noindent {\bf Acknowledgments}: 
This work was supported by the National Natural Science Foundation of China (Grant No.\@ 11771190).
The authors thank the referees for helpful suggestions. 

\smallskip

\noindent {\bf Data availability statement}: 
The data that support the findings of this study can be found in journal publications and the arxiv.

\end{document}